\documentclass[10pt]{amsart}
\usepackage{amsmath}
\usepackage{amsthm}
\usepackage{amssymb}
\usepackage{color}
\usepackage{amscd}
\usepackage{amsfonts}
\usepackage{graphicx}
\usepackage{epsfig}
\usepackage{amssymb,latexsym}
\usepackage{graphicx}

\def\bu{\mathbf{u}}

\def\bv{\mathbf{v}}
\def\rv{\mathrm{v}}
\def\bw{\mathbf{w}}
\def\bU{\mathbf{U}}
\def\bV{\mathbf{V}}
\def\bx{\mathbf{x}}
\def\bz{\mathbf{z}}
\def\by{\mathbf{y}}

\def\bbX{\mathbb{X}}

\def\bL{\mathbf{L}}

\def\bbf{\mathbf{f}}
\def\bbg{\mathbf{g}}

\def\R{\mathbf{R}}

\def\C{\mathbf{C}}

\def\bD{\mathbf{D}}

\def\Mc{\mathbf{M}^{N\times N}(\C)}
\def\F{\mathbf F}

\def\bA{\mathbf A}

\def\bS{\mathbf S}
\def\cC{\mathcal C}

\def\bcero{\mathbf 0}

\newtheorem{theorem}{Theorem}[section]
\newtheorem{corollary}[theorem]{Corollary}
\newtheorem{definition}{Definition}[section]
\newtheorem{lemma}[theorem]{Lemma} 
\newtheorem{proposition}[theorem]{Proposition}

\begin{document}

\title{Young-measure solutions for multi-dimensional systems of conservation laws}
\author{Pablo Pedregal}
\date{} 
\thanks{INEI, U. de Castilla-La Mancha, 13071 Ciudad Real, SPAIN. Supported by MINECO/FEDER grant 
MTM2017-83740-P, by PEII-2014-010-P of the Conserjer\'\i a de
Cultura (JCCM), and by grant GI20152919 of UCLM
}
\begin{abstract}
We explore Young measure solutions of systems of conservation laws through an alternative variational method that introduces a suitable, non-negative error functional to measure departure of feasible fields from being a weak solution. Young measure solutions are then understood as being generated by minimizing sequences for such functional much in the same way as in non-convex, vector variational problems. We establish an existence result for such generalized solutions based on an appropriate structural condition on the system. We finally discuss how the classic concept of a Young measure solution can be improved, and support our arguments by considering a scalar, single equation in dimension one. 
\end{abstract}
\maketitle
\section{Introduction}
We would like to explore a variational approach for weak solutions of systems of conservation laws in high dimension, and assess to what extent such a perspective might be of some help in this field. We will focus on the system
\begin{equation}\label{ecuaini}
\partial_t\bu(t, \bx)+\nabla_\bx\cdot\bbf(\bu(t, \bx))=\bcero\hbox{ in }\Omega\equiv(0, +\infty)\times\R^n,\quad \bu(0, \bx)=\bu_0(\bx).
\end{equation}
The unknown $\bu:(0, +\infty)\times\R^n\to\R^N$ is the vector of conserved quantities, while 
$$
\bbf(\bu):\R^N\to\R^{N\times n}, \quad \bbf(\bu)=(\bbf^{(i)}(\bu))_{i=1, 2, \dots, N}, \bbf^{(i)}(\bu):\R^N\to\R^n,
$$
is the flux. $\bu_0$ is the initial datum. 

Solutions are sought in a weak or integral sense. For definiteness, we assume that the flux $\bbf$ has components of polynomial growth of at most degree $p\ge1$, and the initial datum $\bu_0$ belongs to $L^2(\R^n; \R^N)$. If, in addition, we regard $\bu$ as a vector field in $\bbX=L^2(\Omega; \R^N)\cap L^{2p}(\Omega; \R^N)$, so that the composition $\bbf(\bu(t, \bx))\in L^2(\Omega; \R^{N\times n})$, then test fields $\bw$ can be taken, jointly in time and space variables, in $H^1(\Omega; \R^N)$. 

\begin{definition}\label{original}
A  field $\bu(t, \bx)\in\bbX$ is called a weak solution of \eqref{ecuaini} if
\begin{equation}\label{debil}
\int_\Omega [\bu(t, \bx)\cdot\partial_t \bw(t, \bx)+\bbf(\bu(t, \bx)):\nabla_\bx\bw(t, \bx)]\,dt\,d\bx+\int_{\R^n}\bw(0, \bx)\cdot\bu_0(\bx)\,d\bx=0
\end{equation}
for all test fields $\bw(t, \bx)\in H^1(\Omega; \R^N)$. 
\end{definition}

System \eqref{ecuaini}, or its weak formulation \eqref{debil}, are typically written in the form
\begin{equation}\label{ecuainipri}
\partial_t\bu(t, \bx)+\sum_{j=1}^n\partial_{x_j}\bbf_{(j)}(\bu(t, \bx))=\bcero\hbox{ in }\Omega\equiv(0, +\infty)\times\R^n,\quad \bu(0, \bx)=\bu_0(\bx),
\end{equation}
where
$$
\bbf(\bu):\R^N\to\R^{N\times n}, \quad \bbf(\bu)=(\bbf_{(j)}(\bu))_{j=1, 2, \dots, n}, \bbf_{(j)}(\bu):\R^N\to\R^N.
$$
Similarly, \eqref{debil} becomes
\begin{equation}\label{debilprim}
\int_\Omega [\bu(t, \bx)\cdot\partial_t \bw(t, \bx)+\sum_{j=1}^n\bbf_{(j)}(\bu(t, \bx))\cdot\partial_{x_j}\bw(t, \bx)]\,dt\,d\bx+\int_{\R^n}\bw(0, \bx)\cdot\bu_0(\bx)\,d\bx=0
\end{equation}
for all test fields $\bw(t, \bx)\in H^1(\Omega; \R^N)$. 

The theory of these systems of first-order PDEs is not, apparently, complete. The main issue is the non-uniqueness of weak solutions, as it is not at all clear how to decide, on physical grounds or otherwise, about the ``good" solution.  Entropy criteria have been the universally accepted mechanism to single out the good solution \cite{benzoni-serre}, \cite{bressan}, \cite{EvansB}, \cite{lax}, \cite{tartar0}, \cite{tartar}. However, for systems in several space variables, that criterium does not seem sufficient to select one weak solution. On the other hand, the vanishing viscosity method, where appropriate weak solutions are sought as limits of solutions of the perturbed system
\begin{gather}
\partial_t\bu^{(\epsilon)}(t, \bx)+\nabla_\bx\cdot\bbf(\bu^{(\epsilon)}(t, \bx))-\epsilon\Delta_{\bx\bx}\bu^{(\epsilon)}=\bcero\hbox{ in }\Omega\equiv(0, +\infty)\times\R^n,\nonumber\\
\bu^{(\epsilon)}(0, \bx)=\bu_0(\bx),\nonumber
\end{gather}
also suffers from two main drawbacks: 
\begin{enumerate}
\item there might be more than one branch of solutions $\bu^{(\epsilon)}$;
\item the pointwise convergence, under a uniform bound,  of  a given branch of solutions $\bu^{(\epsilon)}$ to some limit field $\bu$ is very difficult to show in practice.
\end{enumerate}
The first one leads to the fundamental problem of the uniqueness of the ``good" solution; the second one forces to consider the notion of measure-valued solution. 

Our proposal to seek weak solutions of system \eqref{ecuaini} proceeds by examining a certain ``error" functional which tries to measure how far a given field $\bu\in\bbX$ is from being a weak solution of our problem. Namely, for $\bu\in\bbX$, define its associated ``defect" $\bv(t, \bx)\in H^1(\Omega; \R^N)$, as the unique solution of the problem
\begin{equation}
\label{primii}
\int_\Omega[ (\bu+\partial_t\bv)\cdot\partial_t \bw+(\bbf(\bu)+\nabla_\bx \bv):\nabla_\bx\bw]\,dt\,d\bx
+\int_{\R^n} \bu_0(\bx)\cdot\bw(0, \bx)\,d\bx=0
\end{equation}
for all $\bw\in H^1(\Omega; \R^N)$. This variational identity is well-posed as its solution $\bv$ is the unique solution of the  standard variational problem that consists in minimizing, among fields $\bv\in H^1(\Omega; \R^N)$, the quadratic functional
$$
\frac12\int_\Omega [|\partial_t \bv+\bu|^2+|\nabla_\bx \bv+\bbf(\bu)|^2]\,dt\,d\bx+\int_{\R^n}\bu_0(\bx)\cdot\bv(0, \bx)\,d\bx.
$$
Formally, one can also think in terms of the system 
$$
-(\partial_{tt}\bv+\nabla_{\bx\bx}\bv)+\partial_t \bu+\nabla_\bx \bbf(\bu)=\bcero\hbox{ in }\Omega,
$$
together with boundary condition
$$
\partial_t\bv(0, \bx)+\bu(0, \bx)=\bu_0(\bx),\quad \bx\in\R^n.
$$
This viewpoint brings us closer to the least square approach. See \cite{bogunz}, \cite{glowinski}. We will stick however to \eqref{primii} as the neatly specified way to determine the defect $\bv$ for each $\bu\in\bbX$. 

We finally put
\begin{equation}\label{funcerr}
E:\bbX\to\R^+,\quad E(\bu)=\frac12\int_\Omega(|\partial_t \bv|^2+|\nabla_\bx \bv|^2)\,dt\,d\bx,
\end{equation}
and regard $E$ as a ``measure of how far $\bu$ is from being a solution of \eqref{ecuaini}", in the sense that $E(\bu)$ vanishes exactly for solutions of \eqref{ecuaini}. More precisely, if $E(\bu)$ vanishes, and we set $\partial_t\bv=\nabla_\bx\bv=\bcero$ in \eqref{primii}, then we are back to \eqref{debil}
for all $\bw\in H^1(\Omega; \R^N)$, i.e. $\bu$ is a weak solution of \eqref{ecuaini} with initial value $\bu_0$. Our intention is to look at these solutions through the non-local, non-quadratic functional $E$ in \eqref{funcerr}, and try to learn about those solutions through the behavior and properties of the error functional $E$. We therefore focus on looking closely at the functional
$E:\bbX\to\R$ defined through \eqref{funcerr} where $\bv$ is determined through \eqref{primii}. 
The zero set of this functional $E$ is intimately connected to weak solutions of \eqref{ecuaini}. Indeed, if the infimum of $E$ turns out to be strictly positive, we could envision serious difficulties with problem \eqref{ecuaini} as there could not be anything like a weak solution in the space $\bbX$. 

The existence of weak solutions for \eqref{ecuaini} amounts, therefore, to showing
$$
\inf E=\min E=0.
$$
It is not easy to decide where to start in order to show that $\inf E=\min E$. If we resort to classic variational methods, 
we immediately realize that we cannot rely on the traditional facts about functionals of the Calculus of Variations since the functional $E$ is not a typical local, integral functional with a given, specific density. It is more like a cost functional for an optimization problem in the context of distributed parameter systems (\cite{LionsC}) but, we expect, with some special properties because of the way it has been set up.  
Despite these difficulties,  results for existence of minimizers might rely on coercivity, convexity and weak lower semicontinuity, but none of those look easy to prove for $E$ if $\bbf(\bu)$ is non-linear. In fact, the nature of this functional $E$ is such that our insight and intuition with integral functionals may be rather misleading in this context. 

As a possible starting point, one can resort to two main alternatives:
\begin{enumerate}
\item given that the functional $E$ is smooth and non-negative, study the flow through integral curves, and hope that they will take us to weak solutions, or, at least, try to learn something from this possibility;
\item build directly (as discrete approximations in a non-linear Galerkin scenario, or through finite differences for example) minimizing sequences for $E$.
\end{enumerate}
We focus here on the first possibility. 
When caring about the flow of a smooth, bounded-from-below functional $E:\bbX\to\R$, we know that
$$
\inf_{\bu\in\bbX}\|E'(\bu)\|=0
$$
and this infimum is achieved through integral curves. 
However, there is no guarantee that 
$$
\inf_{\bu\in\bbX}E(\bu)=0,
$$
unless we could show that
\begin{equation}\label{errorf}
\lim_{E'(\bu)\to\bcero}E(\bu)=0.
\end{equation}
This condition has been taken as a formal definition of an error functional in \cite{ped1} and \cite{ped2}. At this point property \eqref{errorf} seems to be out-of-reach for the functional in \eqref{funcerr}. Check however the discussion in Section \ref{diserr} below. 

Our main concern here is to examine a situation in which we have a bounded sequence $\{\bu^{(k)}\}\in\bbX$ such that $E'(\bu^{(k)})\searrow\bcero$. Such a sequence may be found, for instance, through a typical descent procedure for $E$. Our main result focuses on guaranteeing under what circumstances such a sequence $\{\bu^{(k)}\}$ may generate a Young-measure solution of \eqref{ecuaini} according to the standard definition. 

\begin{definition}\label{clasica}
A family of probability measures $\nu=\{\nu_{(t, \bx)}\}_{(t, \bx)\in\Omega}$, supported in $\R^N$, 
is a Young-measure solution of \eqref{ecuaini} if \eqref{debil} holds in the average, namely, if 
\begin{equation}\label{debily}
\int_\Omega [\overline\bu(t, \bx)\cdot\partial_t \bw(t, \bx)+\overline\bbf(t, \bx):\nabla_\bx\bw(t, \bx)]\,dt\,d\bx+\int_{\R^n}\bw(0, \bx)\cdot\bu_0(\bx)\,d\bx=0
\end{equation}
for all test fields $\bw(t, \bx)\in H^1(\Omega; \R^N)$, where
\begin{equation}\label{momentos}
\overline\bu(t, \bx)=\int_{\R^N}\bz \,d\nu_{(t, \bx)}(\bz),
\quad \overline\bbf(t, \bx)=\int_{\R^N}\bbf(\bz) \,d\nu_{(t, \bx)}(\bz).
\end{equation}
\end{definition}

Our principal result is an existence theorem for Young-measure solutions, based on some additional property of the error functional $E$ in \eqref{funcerr}. This task motivates the following definition.

\begin{definition}\label{hiper}
System \eqref{ecuaini} is said to be variationally hyperbolic if the flux
$$
\bbf(\bu):\R^N\to\R^{N\times n}, \quad \bbf(\bu)=(\bbf_{(j)}(\bu))_{j=1, 2, \dots, n}, \bbf_{(j)}(\bu):\R^N\to\R^N,
$$
is smooth and for every bounded sequence $\{\bu^{(k)}\}$ in $\bbX$, and every bounded sequence $\{\bv^{(k)}\}$ in $H^1(\Omega; \R^N)$, the condition
$$
\partial_t \bv^{(k)}+\sum_{j=1}^n
\bD\bbf_{(j)}^T(\bu^{(k)})\partial_{x_j}\bv^{(k)}\to\bcero\hbox{ in }L^2(\Omega; \R^N)
$$
implies $\bv^{(k)}\to\bcero$ in $L^2(\Omega; \R^N)$.
\end{definition}

Our main existence result follows.

\begin{theorem}\label{ppt}
Let system \eqref{ecuaini} be variationally hyperbolic. If the flow of the corresponding functional $E$ in \eqref{funcerr} produces a sequence of fields $\{\bu^{(k)}\}$, i.e. $E'(\bu^{(k)})\to\bcero$, such that $\{\bu^{(k)}\}$, $\{\bbf(\bu^{(k)})\}$ are weakly convergent (for instance if $\{\bu^{(k)}\}$ is uniformly bounded), then (a suitable subsequence of) this sequence of fields
generates a Young measure solution. 
\end{theorem}

Depending on whether $\{\bu^{(k)}\}$ converges strongly or just weakly to $\bu$, the Young measure solution will be a classic weak solution $\bu$ of the problem, or a true Young measure solution with first moment $\bu$, respectively. 

Another important job is to identify more manageable conditions on the fluxes $\bbf_{(j)}$ to ensure variational hyperbolicity. We explore two different situations. 
\begin{enumerate}
\item One refers to Friedrichs symmetrizable systems in the following precise sense (see \cite{benzoni-serre}). A family of operators
\begin{equation}\label{famop}
\bL_\bu=\partial_t+\sum_{j=1}^n \bA_{(j)}(\bu(t, \bx))\partial_{x_j}
\end{equation}
for each field $\bu\in\bbX$, admits a symbolic symmetrizer if there is a smooth mapping
$$
\bS(\bu, \xi):\R^N\times(\R^n\setminus\{\bcero\})\to\Mc
$$
homogeneous of degree zero in $\xi$, with the property
$$
\bS(\bu, \xi)=\bS(\bu, \xi)^*>\bcero,\quad \bS(\bu, \xi)\bA(\bu, \xi)=\bA(\bu, \xi)^*\bS(\bu, \xi),
$$
where
$$
\bA(\bu, \xi)=\sum_{j=1}^n\xi_j \bA_{(j)}(\bu).
$$
\item The other one is related to some old facts about the ill-posedness of the Cauchy problem for linear, constant-coefficients systems in  $L^p$ for $p\neq2$. Indeed Brenner (\cite{brenner1}, \cite{brenner2}) showed that the Cauchy problem for systems of the form
$$
\partial_t+\sum_{j=1}^n\bA_j\partial_{x_j},\quad \bA_j,\hbox{ constant, $N\times N$-matrices},
$$
is ill-posed in $L^p$ for $p\neq2$ except in the case where matrices $\bA_j$ commute with each other. 
\end{enumerate}

Though we have not been able to prove variational hyperbolicity in the first case (we are pretty close to do so), we show however the following fact.

\begin{proposition}\label{varhy}
System \eqref{ecuaini} is variationally hyperbolic 
if the differential of the fluxes $\bbf_{(j)}$ commute with each other
$$
\bD\bbf_{(j)}(\bu)\,\bD\bbf_{(k)}(\bu)=\bD\bbf_{(k)}(\bu)\,\bD\bbf_{(j)}(\bu)
$$
for all $j, k$ and $\bu\in\R^N$.
\end{proposition}

As a direct consequence of Theorem \ref{ppt} and Proposition \ref{varhy}, we can state immediately the following corollary.

\begin{corollary}
Assume that the fluxes $\bbf_{(j)}$ in system \eqref{ecuaini} are smooth, and their differentials $\bD\bbf_{(j)}$ commute with each other according to Proposition \ref{varhy}. If the flow of $E$ produces a bounded sequence of fields $\{\bu^{(k)}\}$ in $\bbX$, then (a suitable subsequence of) this sequence of fields generates a Young measure solution of \eqref{ecuaini} according to Definition \eqref{clasica}.
\end{corollary}

In Section \ref{young}, we review some basic definitions and concepts related to Young measure theory and Young measure solutions for \eqref{ecuaini} as in Definition \ref{clasica}. We next motivate the definition of variational hyperbolicity relating it to the derivative $E'$ of the functional $E$, and explore a couple of scenarios about the structure of the initial system \eqref{ecuaini} and how this structure may or may not lead to the variational hyperbolicity property. 
In particular, Section \ref{tres} focuses on the proofs of  Theorem \ref{ppt} and Proposition \ref{varhy}. 
In Section \ref{diserr}, we also argue why we believe Definition \ref{clasica} must be strengthened, and support our critique by looking at the case of a single equation in dimension one where many ingredients can be made fully explicit including numerical approximations. 

\section{Young-measure solutions}\label{young}
There is a lot of recent interest in measure-valued or, more explicitly, Young measure-valued solutions. They are usually defined insisting in that they are solutions of system \eqref{ecuaini} in an average sense. More formally, we have the following universally accepted definition \cite{diperna}, \cite{dst}, \cite{feireisl}, \cite{tadmor2}, \cite{tadmor}, \cite{mnrr}, \cite{neustupa}, \cite{szekelyhidi}, \cite{tartar}. 

\begin{definition}\label{sec}
A family of probability measures $\nu=\{\nu_{(t, \bx)}\}_{(t, \bx)\in\Omega}$, supported in $\R^N$, 
is a Young-measure solution of \eqref{ecuaini} if \eqref{debil} holds in the average, namely, if 
\begin{equation}\label{debily}
\int_\Omega [\overline\bu(t, \bx)\cdot\partial_t \bw(t, \bx)+\overline\bbf(t, \bx):\nabla_\bx\bw(t, \bx)]\,dt\,d\bx+\int_{\R^n}\bw(0, \bx)\cdot\bu_0(\bx)\,d\bx=0
\end{equation}
for all test fields $\bw(t, \bx)\in H^1(\Omega; \R^N)$, where
\begin{equation}\label{momentos}
\overline\bu(t, \bx)=\int_{\R^N}\bz \,d\nu_{(t, \bx)}(\bz),
\quad \overline\bbf(t, \bx)=\int_{\R^N}\bbf(\bz) \,d\nu_{(t, \bx)}(\bz).
\end{equation}
\end{definition}

It is pretty clear that conventional weak solutions $\bu(t, \bx)$ correspond to Young-measure solutions that in fact are atomic or Dirac delta masses
$$
\nu_{(t, \bx)}=\delta_{\bu(t, \bx)}.
$$

We recall, for the convenience of readers, some fundamental facts about Young measures. 

\begin{definition}\label{YM}
A bounded sequence of fields $\{\bw_j(\by)\}$ in $L^p(D; \R^d)$, 
$$
\bw_j:D\subset\R^N\to\R^d,\quad p\ge1,
$$
generates the Young measure $\mu=\{\mu_\by\}_{\by\in D}$, where each $\mu_\by$ is a probability measure supported in $\R^d$, if whenever the composition $\{\Psi(\bw_j(\by))\}$ converges weakly (in $L^1(D)$) to some $\overline\Psi(\by)$, for a given continuous function $\Psi$, we have
\begin{equation}\label{represen}
\overline\Psi(\by)=\int_{\R^d}\Psi(\bz)\,d\nu_\by(\bz).
\end{equation}
\end{definition}

The most remarkable fact about the Young measure is that the family of probability measures $\{\mu_\by\}$ is determined by the sequence of fields $\{\bw_j\}$, and must not change with $\Psi$: that same family of probability measures furnishes all weak limits of compositions with arbitrary functions $\Psi$ through \eqref{represen}. Put it explicitly
$$
\Psi(\bw_j(\by))\rightharpoonup \int_{\R^d}\Psi(\bz)\,d\nu_\by(\bz)
$$
for every continuous function $\Psi$ for which $\{\Psi(\bw_j(\by))\}$ is weakly convergent. 

The most basic existence theorem for Young measures is the following \cite{ball}, \cite{ped4}.
\begin{theorem}
Let $\{\bw_j(\by)\}$ be a uniformly bounded sequence of fields as in Definition \ref{YM}. Then there is a subsequence, not relabeled, and a family of probability measures $\mu=\{\mu_\by\}_{\by\in D}$ such that \eqref{represen} holds as soon as $\{\Psi(\bw_j(\by))\}$ is weakly convergent for a continuous $\Psi$. In this case, we will write $\{\bw_j\}\sim\mu$. 
\end{theorem}

What this result is guaranteeing is that bounded sequences of fields always generate, module subsequences, Young measures. If there are no constraint to be respected on fields or on probability measures, then the converse is also correct. 

\begin{theorem}\label{genera}
Let $\mu=\{\mu_\by\}_{\by\in D}$ be a family of probability measures, supported in $\R^d$, such that
$$
\int_D\int_{\R^d}|\bz|^p\,d\mu_\by(\bz)\,d\by<+\infty, \quad p\ge1.
$$ 
Then there are uniformly bounded sequences of fields $\{\bw_j(\by)\}$ in $L^p(D; \R^d)$ whose underlying Young measure, according to Definition \ref{YM}, is precisely  the given family $\mu=\{\mu_\by\}_{\by\in D}$. 
\end{theorem}

We can now establish the relationship of a Young measure solution to the defects of generating sequences. 

\begin{proposition}\label{yms}
A family of probability measures $\nu$ is a Young-measure solution according to Definition \ref{sec}, generated by a sequence $\{\bu_j\}$  ($\{\bu_j\}\sim\nu$) such that $\{\bu_j\}$, and $\{\bbf(\bu_j)\}$ are weakly convergent sequences in $L^2(\Omega; \R^N)$, if and only if for the sequence of associated defects $\bv_j$ in \eqref{primii}, we have that $\bv_j\rightharpoonup\bcero$ in $H^1(\Omega; \R^N)$.
\end{proposition}
\begin{proof}
The proof is nothing but a comparison between the two variational identities
$$
\int_\Omega [\overline\bu(t, \bx)\cdot\partial_t \bw(t, \bx)+\overline\bbf(t, \bx):\nabla_\bx\bw(t, \bx)]\,dt\,d\bx+\int_{\R^n}\bw(0, \bx)\cdot\bu_0(\bx)\,d\bx=0,
$$
and
$$
\int_\Omega[ (\bu_j+\partial_t\bv_j)\cdot\partial_t \bw+(\bbf(\bu_j)+\nabla_\bx \bv_j):\nabla_\bx\bw]\,dt\,d\bx+\int_{\R^n} \bu_0(\bx)\cdot\bw(0, \bx)\,d\bx=0
$$
for all $\bw\in H^1(\Omega; \R^N)$. Because
$$
\bu_j\rightharpoonup\overline\bu,\quad \bbf(\bu_j)\rightharpoonup\overline\bbf,
$$
we can only conclude that
$$
\int_\Omega(\partial_t\bv_j\cdot\partial_t \bw+\nabla_\bx \bv_j:\nabla_\bx\bw
+ \bv_j\cdot\bw)\,dt\,d\bx=0
$$
for all $\bw\in H^1(\Omega; \R^N)$. 
\end{proof}

One relevant remark is that we have explicitly discarded concentration effects created by the generating sequence $\{\bu_j\}$ in Definition \ref{nueva} given that it is well-known that Young measure are not tailored to capture this phenomenon. A bit more sophisticated concept of measure-valued solution can be implemented to account also for concentrations in generating sequences of nearly solutions (see \cite{szekelyhidi} for instance). 

It is interesting to extend the error functional $E$ for Young measures generated by arbitrary sequences in the space $\bbX$ in the following way. Put 
$$
\overline\bbX=\{\nu=\{\nu_{(t, \bx)}\}_{(t, \bx)\in\Omega}: \{\bu_j\}\sim\nu, \{\bu_j\}\subset\bbX\},
$$
we can define
$$
\overline E:\overline\bbX\to\R,\quad \overline E(\nu)=\frac12\int_\Omega(|\partial_t \bv|^2+|\nabla_\bx \bv|^2)\,dt\,d\bx,
$$
where, for 
$$
\overline\bu(t, \bx)=\int_{\R^N}\bz \,d\nu_{(t, \bx)}(\bz),
\quad \overline\bbf(t, \bx)=\int_{\R^N}\bbf(\bz) \,d\nu_{(t, \bx)}(\bz).
$$
the defect field $\bv\in H^1(\Omega; \R^N)$ is uniquely determined by
$$
\int_\Omega[ (\overline\bu+\partial_t\bv)\cdot\partial_t \bw+(\overline\bbf+\nabla_\bx \bv):\nabla_\bx\bw]\,dt\,d\bx+\int_{\R^n} \bu_0(\bx)\cdot\bw(0, \bx)\,d\bx=0
$$
for all $\bw\in H^1(\Omega; \R^N)$. It is easy to realize that Young measures minimizers would correspond to Young measure solutions according to Definition \ref{sec}. 

\section{Variational hyperbolicity}\label{tres}
Let us start by computing the derivative of $E$. To this aim, we become interested in the limit
$$
\langle E'(\bu), \bU\rangle=\lim_{\epsilon\to0}\frac{E(\bu+\epsilon \bU)-E(\bu)}\epsilon,
$$
for $\bu$, $\bU$ in $\bbX$, and where $E'(\bu)$ is understood as an element of the dual of $\bbX$. Let us put $\bv+\epsilon \bV$ for the perturbation produced in the defect $\bv$ of $\bu$, when we perturb $\bu$ by $\bU$. Then 
\begin{gather}
\int_\Omega[ (\bu+\epsilon \bU+\partial_t \bv+\epsilon\partial_t \bV)\cdot\partial_t \bw+(\bbf(\bu+\epsilon \bU)+\nabla_\bx \bv+\epsilon\nabla_\bx \bV):\nabla_\bx \bw]\,dt\,d\bx\nonumber\\
+\int_{\R^n} \bu_0(\bx)\cdot\bw(0, \bx)\,d\bx=0.\nonumber
\end{gather}
Even more explicitly
\begin{gather}
\int_{\R^n} \bu_0(\bx)\cdot\bw(0, \bx)\,d\bx+\int_\Omega(\bu+\epsilon \bU+\partial_t \bv+\epsilon\partial_t \bV)\cdot\partial_t \bw\,dt\,d\bx\nonumber\\
+\int_\Omega\sum_{j=1}^n(\bbf_{(j)}(\bu+\epsilon \bU)+\partial_{x_j} \bv+\epsilon\partial_{x_j} \bV)\cdot\partial_{x_j} \bw\,dt\,d\bx=0.\nonumber
\end{gather}
By differentiating with respect to $\epsilon$, and setting $\epsilon=0$ afterwards, we find
\begin{equation}\label{derivada}
\int_\Omega[ (\bU+\partial_t \bV)\cdot\partial_t\bw+\sum_{j=1}^n
\partial_{x_j}\bw\,\bD \bbf_{(j)}(\bu)\bU+\nabla_\bx \bV:\nabla_\bx\bw]\,dt\,d\bx=0,
\end{equation}
which is regarded as a variational identity determining $\bV$, once $\bu$, $\bv$, and $\bU$ are known. 
Note that each $\bD\bbf_{(j)}(\bu)$ is a $N\times N$ matrix. 
On the other hand, the derivative of the error itself is
$$
\langle E'(\bu), \bU\rangle=\int_\Omega(\nabla_\bx\bv\cdot\nabla_\bx\bV+\partial_t\bv\cdot\partial_t\bV)\,dt\,d\bx,
$$
and by taking $\bw=\bv$ in \eqref{derivada}, we have
$$
\langle E'(\bu), \bU\rangle=-\int_\Omega \bU\cdot(\partial_t\bv+\sum_{j=1}^n
\bD\bbf_{(j)}^T(\bu)\partial_{x_j}\bv)\,dt\,d\bx.
$$
We, therefore, clearly see, due to the arbitrariness of $\bU$, that
$$
E'(\bu)=-(\partial_t \bv+\sum_{j=1}^n
\bD\bbf_{(j)}^T(\bu)\partial_{x_j}\bv),
$$
where $\bv$ is the unique solution in \eqref{primii}. 

According to Definition \ref{hiper}, if system \eqref{ecuaini} is variationally hyperbolic and $\{\bu^{(k)}\}$ is a bounded sequence in $\bbX$ with corresponding defects $\{\bv^{(k)}\}$ such that $E'(\bu^{(k)})\to\bcero$, then $\bv^{(k)}\to\bcero$ in $L^2(\Omega; \R^N)$, though this convergence might not take place in $H^1(\Omega; \R^N)$. However, as discussed in Section \ref{young}, convergence to zero in $L^2(\Omega; \R^N)$ is all that is required for a Young-measure solution. 

We are ready to prove Theorem \ref{ppt}. 

\begin{proof}[Proof of Theorem \ref{ppt}]
Suppose $\{\bu^{(k)}\}$ is a bounded sequence in $\bbX$ such that both $\{\bu^{(k)}\}$ and $\{\bbf(\bu^{(k)})\}$ are weakly convergent in $\bbX$, and $E'(\bu^{(k)})\to\bcero$. If we use $\bx=\bv^{(k)}$ as a test function in \eqref{primii}, we find
$$
\|\bv^{(k)}\|^2_{H^1(\Omega; \R^N)}=-\int_\Omega (\bu^{(k)}\cdot\partial_t\bv^{(k)}+\bbf(\bu^{(k)}):\nabla_\bx\bv^{(k)})\,dt\,d\bx-\int_{\R^n}\bu_0(\bx)\cdot\bv^{(k)}(0, \bx)\,d\bx,
$$
and so
$$
\|\bv^{(k)}\|_{H^1(\Omega; \R^N)}\le \max(\|\bu^{(k)}\|_{L^2(\Omega; \R^N)}, \|\bbf(\bu^{(k)})\|_{L^2(\Omega; \R^N)})
+\|\bu_0\|_{L^2(\R^n)}.
$$
In this way, the sequence of defects $\{\bv^{(k)}\}$ is uniformly bounded in $H^1(\Omega; \R^N)$. Definition \eqref{hiper} implies then that, at least for a suitable subsequence, $\bv^{(k)}\to\bcero$ weakly in $H^1(\Omega; \R^N)$ and strongly in $L^2(\Omega; \R^N)$. According to Proposition \ref{yms}, the Young measure generated by (an appropriate subsequence of ) $\{\bu^{(k)}\}$ becomes a Young measure solution of our initial system. 
\end{proof}

We would like to better understand variational hyperbolicity, and, in particular, provide a proof of Proposition \ref{varhy}. 
The analysis of linear, variable coefficients, first-order PDE systems of the kind 
\begin{equation}\label{hfos}
\partial_t \bv+\sum_{j=1}^n
\bD\bbf_{(j)}^T(\bu)\partial_{x_j}\bv=\F
\end{equation}
require regularity of the coefficients $\bD\bbf_{(j)}^T(\bu)$, which, in general, is not the case if the field $\bu$ only belong to $\bbX$. 
Notice, however, that we are not facing the task of proving existence of solutions for \eqref{hfos}, but rather some kind of a priori analysis or energy estimates. We are given the sequence of fields $\bv^{(k)}$ complying with \eqref{hfos}
\begin{equation}\label{hfoss}
\partial_t \bv^{(k)}+\sum_{j=1}^n
\bD\bbf_{(j)}^T(\bu^{(k)})\partial_{x_j}\bv^{(k)}=\F^{(k)}\to\bcero\hbox{ in }L^2(\Omega; \R^N),
\end{equation}
and would like to conclude that they necessarily converge to zero in $L^2(\Omega; \R^N)$. A natural strategy is to use energy estimates for this kind of systems. We resort to some basic material from \cite{benzoni-serre}. 

A family of operators
$$
\bL_\bu=\partial_t+\sum_{j=1}^n \bA_{(j)}(\bu(t, \bx))\partial_{x_j}
$$
for each field $\bu\in\bbX$, admits a symbolic symmetrizer if there is a smooth mapping
$$
\bS(\bu, \xi):\R^N\times(\R^n\setminus\{\bcero\})\to\Mc
$$
homogeneous of degree zero in $\xi$, with the property
$$
\bS(\bu, \xi)=\bS(\bu, \xi)^*>\bcero,\quad \bS(\bu, \xi)\bA(\bu, \xi)=\bA(\bu, \xi)^*\bS(\bu, \xi),
$$
where
$$
\bA(\bu, \xi)=\sum_{j=1}^n\xi_j \bA_{(j)}(\bu).
$$
Theorem 2.5 in \cite{benzoni-serre} reads:
\begin{theorem}\label{bense}
Suppose $\bA(\bu, \xi)$ admits a symbolic symmetrizer and that
$$
\|\bu\|_{W^{1, \infty}([0, T]\times\R^n; \R^N)}\le \omega.
$$
Then there are constants $K=K(\omega)$ and $\gamma=\gamma(\omega)$ in such a way that 
$$
\|\bv(\cdot, t)\|^2_{L^2(\R^n; \R^N)}\le K\left(e^{\gamma t}\|\bv(\cdot, 0)\|^2_{L^2(\R^n; \R^N)}+
\int_0^te^{\gamma(t-\tau)}\|\bL_\bu\bv(\tau)\|^2_{L^2(\R^n; \R^N)}\,d\tau\right),
$$
for all $\bv\in\cC^1([0, T]; L^2(\R^n; \R^N))\cap\cC([0, T]; H^1(\R^n; \R^N))$.
\end{theorem}
However, to fit our purposes here, we need a slightly different version of this result which is also contained in \cite{benzoni-serre}: Remarks 2.1 and 2.4. The first one refers to reversing time, and adjusting the corresponding linear operator; the second one involves a standard density argument. 

\begin{corollary}\label{bensedos}
Suppose $\bA(\bu, \xi)$ admits a symbolic symmetrizer and that
$$
\|\bu\|_{W^{1, \infty}([0, T]\times\R^n; \R^N)}\le \omega,
$$
for $T>0$. 
Then there are constants $K=K(\omega)$ and $\gamma=\gamma(\omega)$ in such a way that,  
$$
\|\bv(\cdot, t)\|^2_{L^2(\R^n; \R^N)}\le K e^{\gamma(T-t)}\left(\|\bv(\cdot, T)\|^2_{L^2(\R^n; \R^N)}+
\int_0^T\|\bL_\bu\bv(\tau)\|^2_{L^2(\R^n; \R^N)}\,d\tau\right),
$$
for all $\bv\in H^1([0, T]\times\R^n)$, and every $t\in[0, T]$.
\end{corollary}

This result is not fine enough to show variational hyperbolicity because of the uniform bound $\omega$ required on the uniform Lipschitz norm, for which one can hardly find a substitute if your sequence $\{\bu^{(k)}\}$ is just bounded in $L^2(\Omega; \R^N)$. 

\begin{lemma}
Consider our initial system \eqref{ecuaini} with a smooth flux $\bbf$ such that the family of operators
$$
\bL_\bu=\partial_t+\sum_{j=1}^n \bD\bbf_{(j)}^T(\bu)\partial_{x_j}
$$
admits a symbolic symmetrizer. If $\{\bu^{(k)}\}$ is uniformly bounded in $W^{1, \infty}(\Omega; \R^N)$, and $\{\bv^{(k)}\}$ is bounded in $H^1(\Omega; \R^N)$ with
\begin{equation}\label{hfosss}
\bL_{\bu^{(k)}}\bv^{(k)}\to\bcero
\end{equation}
in $L^2(\Omega; \R^N)$, then $\bv^{(k)}\to\bcero$ in $\bbX$.
\end{lemma}
\begin{proof}

Suppose that $\{\bu^{(k)}\}$ is unformly bounded in $\bbX$; $\{\bv^{(k)}\}$, uniformly bounded in $H^1(\Omega; \R^N)$; and \eqref{hfosss} holds. Without loss of generality, we can take for granted that $\bv^{(k)}$ converges weakly in $H^1(\Omega; \R^N)$ and strongly in $L^2(\Omega; \R^N)$ to some $\bv\in H^1(\Omega; \R^N)$. Under our assumptions, the constants $K$ and $\gamma$ of Corollary \ref{bensedos} do not depend on $k$. 

Take $\epsilon>0$.  Select $T=T(\epsilon)$ large enough so that 
$$
Ke^\gamma\|\bv(\cdot, T)\|^2_{L^2(\R^n; \R^N)}\le \epsilon.
$$
Corollary \ref{bensedos} lets us write
\begin{align}
\|\bv(\cdot, t)^{(k)}\|^2_{L^2(\R^n; \R^N)}\le K e^{\gamma(T-t)}&\left(\|\bv^{(k)}(\cdot, T)\|^2_{L^2(\R^n; \R^N)}\right.\nonumber\\
&\left.+\int_0^\infty\|\bL_{\bu^{(k)}}\bv^{(k)}(\tau)\|^2_{L^2(\R^n; \R^N)}\,d\tau\right).\nonumber
\end{align}
For $t\in[T-1, T]$ and $k$ large enough (depending on $\epsilon$), due to \eqref{hfosss},
$$
\|\bv(\cdot, t)^{(k)}\|^2_{L^2(\R^n; \R^N)}\le K e^\gamma\|\bv(\cdot, T)\|^2_{L^2(\R^n; \R^N)}+\epsilon\le2\epsilon.
$$
Hence we can conclude that $\bv=\bcero$  in $[T-1, T]\times\R^n$. The arbitrariness of $T$, implies that $\bv\equiv\bcero$ and $\bv^{(k)}\to\bcero$ in $\bbX$.
\end{proof}

Our main result on variational hyperbolicity follows. Note how this statement holds trivially when either of the two dimensions $N$ or $n$ is unity, so that it is meaningful for systems in high dimension. 
\begin{proposition}
If the set of fluxes $\bbf_{(j)}$ are such that their differentials commute with each other
\begin{equation}\label{conmutacion}
\bD\bbf_{(j)}(\bu)\,\bD\bbf_{(k)}(\bu)=\bD\bbf_{(k)}(\bu)\,\bD\bbf_{(j)}(\bu)
\end{equation}
for all $j, k$ and $\bu\in\R^N$, then system \eqref{ecuaini} is variationally hyperbolic.
\end{proposition}

\begin{proof}
We proceed in four steps. 

1. The scalar case in dimension 1. 
Suppose that $\{u^{(j)}\}$, a bounded sequence in $\bbX$; $F_j\to0$ in $L^2(\Omega)$; and $\{v^{(j)}\}\subset H^1(\Omega)$, a bounded sequence, are such that
$$
v^{(j)}_t+f'(u^{(j)})v^{(j)}_x+g(t, x, u^{(j)})v^{(j)} =F_j\hbox{ in }\Omega,
$$
for certain smooth functions $f'$ and $g$. If $v^{(j)}\to v$ in $L^2(\Omega)$ (possibly for a suitable subsequence), 
we would like to conclude that in fact $v\equiv0$, and $v^{(j)}\to0$ in $L^2(\Omega)$.

Since, on the other hand, the sequence $\{f'(u^{(j)})\}$ is uniformly bounded in $L^2(\Omega)$, we can always, due to the density of smooth functions in $L^2(\Omega)$, find a certain $U^{(j)}\in W^{1, \infty}(\Omega)$ such that
$$
\|U^{(j)}\|_{W^{1, \infty}(\Omega)}\to\infty,\quad \|U^{(j)}-f'(u^{(j)})\|_{L^2(\Omega)}\to0.
$$
Same argument applies to $\{g(t, x, u^{(j)}\}$. 
Hence, we can assume, without loss of generality, that
\begin{equation}\label{modificado}
v^{(j)}_t+U^{(j)}v^{(j)}_x+V^{(j)}v^{(j)}=\overline F_j\hbox{ in }\Omega,
\end{equation}
where $\{U^{(j)}\}$ and  $\{V^{(j)}\}$ are sequences of Lipschitz functions, uniformly bounded in $L^2(\Omega)$, and $\overline F_j\to0$ in $L^2(\Omega)$. 

We are now entitled to use characteristics, which are the solutions $x^{(j)}$ of
$$
(x^{(j)})'(s)=U^{(j)}(s, x^{(j)}(s))\quad s>0.
$$
Equation \eqref{modificado} implies that 
\begin{equation}\label{deriv}
\frac d{ds}v^{(j)}(s, x^{(j)}(s))+V^{(j)}v^{(j)}(s, x^{(j)}(s))=v^{(j)}_t+U^{(j)}v^{(j)}_x+V^{(j)}v^{(j)}\to0
\end{equation}
as $j\to\infty$ for all $s$. This condition means, if we set $\overline v^{(j)}(s)=v^{(j)}(s, x^{(j)}(s))$, that
$$
\frac d{ds}\overline v^{(j)}+V^{(j)}\overline v^{(j)}\to0
$$
as $j\to\infty$ for all $s$. 
The explicit formula for the solution of non-homogeneous, linear, first-order ODE (for instance) allows us to conclude that 
$$
v^{(j)}(t, x^{(j)}(t))\to0\hbox{ for all }t,
$$
if for some fixed $s$,
$$
v^{(j)}(t, x^{(j)}(s))\to0.
$$
Take an arbitrary fixed time $t$, and two characteristics $x^{(j)}$ and $y^{(j)}$. For $s>t$, 
\begin{align}
v^{(j)}(t, x^{(j)}(t))-v^{(j)}(t, y^{(j)}(t))=\,\,&
v^{(j)}(t, x^{(j)}(t))-v^{(j)}(s, x^{(j)}(s))\nonumber\\
&+v^{(j)}(s, x^{(j)}(s))-v(s, x^{(j)}(s))\nonumber\\
&+v(s, x^{(j)}(s))-v(s, y^{(j)}(s))\nonumber\\
&+v(s, y^{(j)}(s))-v^{(j)}(s, y^{(j)}(s))\nonumber\\
&+v^{(j)}(s, y^{(j)}(s))-v^{(j)}(t, y^{(j)}(t)).\nonumber
\end{align}
Because $v\in L^2(\Omega)$, select $s$ sufficiently large so that the third term in this sum is small, independently of $j$. Because $v^{(j)}\to v$ in $L^2(\Omega)$, take $j$ large enough to have that the second and fourth terms are small. By the prior discussion, once $v^{(j)}(s, y^{(j)}(s))$ and $v^{(j)}(s, x^{(j)}(s))$ are small, then both 
$$
v^{(j)}(t, x^{(j)}(t)),\quad v^{(j)}(t, x^{(j)}(t))
$$
are also small for $j$ large. Altogether we see that 
$$
v^{(j)}(t, x^{(j)}(t))-v^{(j)}(t, y^{(j)}(t))\to0\hbox{ as }j\to\infty,
$$
and because of the arbitrariness of $t$ and the two characteristics considered, given that $v^{(j)}\in L^2(\Omega)$, we can conclude that 
$$
v^{(j)}(t, x)\to0
$$
for a.e. $(t, x)$. This implies that the limit function $v$ vanishes identically, and then $v^{(j)}\to0$ in $L^2(\Omega)$. 

2. The scalar case in dimension $n$. This is formally the same as the previous case. The equation would be
$$
v^{(k)}_t+\sum_j f'_{(j)}(u^{(k)})v^{(k)}_{x_j}+g(t, \bx, u^{(k)})v^{(k)} =F_k\hbox{ in }\Omega,
$$
and we can use characteristics in just the same way as before. 

3. The vector case in dimension 1. This time we have
$$
\bv^{(k)}_t+\bD\bbf^T(\bu^{(k)})\bv_x+\bbg(t, x, \bu^{(k)})=\F_k,
$$
with $\bD\bbf^T(\bu)$ diagonalizedable
$$
\bD\bbf^T(\bu)\rv_l(\bu)=\lambda_l(\bu)\rv_l(\bu).
$$
It is standard that if we put
$$
\bv^{(k)}(t, x)=\sum_{l=1}^N v_l^{(k)}(t, x)\rv_l(\bu^{(k)}),
$$
then decomposition of all fields in the basis $\{\rv_l\}$ leads exactly to a family of decoupled scalar equations in dimension 1 for each coefficient $v_l^{(k)}$ of the type examined in the first step. We can conclude then that $v_l^{(k)}\to0$ for each $l$ as $k\to\infty$. This leads to $\bv^{(k)}\to\bcero$ as well.

4. The multi-dimensional, vector case. It is an elementary fact in Linear Algebra that the commutation condition \eqref{conmutacion} implies that there exists a common basis of eigenvectors $\rv_l(\bu)$, $l=1, 2, \dots, N$, for all $N\times N$-matrices $\bD\bbf_{(j)}(\bu)$. In this case, it is also standard to check that system \eqref{hfos} decouples into $N$ scalar equations in dimension $n$ if we write, as in the previous step,
$$
\bv(t, \bx)=\sum_{l=1}^N v_l(t, \bx)\rv_l(\bu)
$$
then each $v_l(t, \bx)$ is the solution of a scalar equation of the type explored in Step 2. We again conclude that $\bv^{(k)}\to\bcero$ as $k\to\infty$.
\end{proof}

\section{Critique to the concept of  Young-measure solution}\label{diserr}
We would like to argue that Definition \ref{sec} of a Young-measure valued solution ought to be improved. The intuitive reason is that knowing or fixing a finite number of moments of a probability measure, as expressed in \eqref{momentos}, does not say much about the underlying probability measures themselves. Indeed, \eqref{debily} and \eqref{momentos} does not impose much about the Young-measure valued solution $\nu$. Even more so, one could change drastically the whole family $\nu$, as long as preserving the integrals \eqref{momentos}, and we would have another such Young-measure solution. 
Indeed, we could find a pair of fields $(\overline\bu, \overline\bbf)$ verifying \eqref{debily}, and from them all families of probability measures furnishing the representation in \eqref{momentos}. All those would be Young measure solutions. 
We believe this is not appropriate because the structure of $\nu$, or its support, does not seem to play a role as far as the system \eqref{ecuaini} is concerned. 

Our proposal for a new definition of Young-measure solutions requires to recall the original concept of a Young measure as it relates to sequences of generating functions or fields, just as we have done in Section \ref{young}. 
Our idea is that a Young measure solution of \eqref{ecuaini} must be a family of probability measures, a Young measure, that can be generated, as in Theorem \ref{genera}, and possibly among many other possibilities, by a sequence of nearly solutions of \eqref{ecuaini} in the following sense. 
\begin{definition}\label{nueva}
A family of probability measures $\nu=\{\nu_{(t, \bx)}\}_{(t, \bx)\in\Omega}$, supported in $\R^N$, 
is a strong Young-measure solution of \eqref{ecuaini} if it can be generated, as the Young measure, by a sequence of fields $\{\bu^{(k)}\}$ such that $E(\bu^{(k)})\searrow0$, and both sequences $\{\bu^{(k)}\}$, $\{\bbf(\bu^{(k)})\}$ are weakly convergent in $L^2(\Omega; \R^N)$ and $L^2(\Omega; \R^{N\times n})$, respectively. 
\end{definition}

Note how the requirement $E(\bu^{(k)})\searrow0$ is a precise way of saying that $\{\bu^{(k)}\}$ is a sequence of nearly solutions of  \eqref{ecuaini}. 
This definition is coherent with Definition \ref{sec}. 
\begin{proposition}\label{impli}
Let $\nu=\{\nu_{(t, \bx)}\}_{(t, \bx)\in\Omega}$ be a strong Young measure-valued solution according to Definition \ref{nueva}. Then it is also a measure-valued solution. 
\end{proposition}
\begin{proof}
Our Definition \ref{nueva} implies the existence of a sequence of fields $\bu^{(k)}$ such that 
$$
\int_\Omega[ (\bu^{(k)}+\partial_t\bv^{(k)})\cdot\partial_t \bw+(\bbf(\bu^{(k)})+\nabla_\bx \bv^{(k)}):\nabla_\bx\bw]\,dt\,d\bx+\int_{\R^n} \bu_0(\bx)\cdot\bw(0, \bx)\,d\bx=0
$$
for all $\bw\in H^1(\Omega; \R^N)$ in such a way that $\bv^{(k)}\to\bcero$ in $H^1(\Omega; \R^N)$. Recall that 
$$
E(\bu^{(k)})=\frac12\|\bv^{(k)}\|^2_{H^1(\Omega; \R^N)}.
$$
Taking limits in this identity, and bearing in mind the hypotheses on the  weak convergence of $\{\bu^{(k)}\}$, $\{\bbf(\bu^{(k)})\}$, we can directly conclude, given the representation formula of weak limits in terms of the underlying Young measure, that 
\eqref{debil} holds in the average
$$
\int_\Omega [\overline\bu(t, \bx)\cdot\partial_t \bw(t, \bx)+\overline\bbf(t, \bx):\nabla_\bx\bw(t, \bx)]\,dt\,d\bx+\int_{\R^n}\bw(0, \bx)\cdot\bu_0(\bx)\,d\bx=0
$$
for all test fields $\bw(t, \bx)\in H^1(\Omega; \R^N)$, where
$$
\overline\bu(t, \bx)=\int_{\R^N}\bz \,d\nu_{(t, \bx)}(\bz),
\quad \overline\bbf(t, \bx)=\int_{\R^N}\bbf(\bz) \,d\nu_{(t, \bx)}(\bz).
$$
After all, $\bv^{(k)}\to\bcero$ in $H^1(\Omega; \R^N)$ obviously implies that $\bv^{(k)}\rightharpoonup\bcero$ in $H^1(\Omega; \R^N)$ as required by Proposition \ref{yms}. 
\end{proof}
The important point we would like to stress is that our Definition \ref{nueva} demands, for a Young measure solution, to be generated by at least one  sequence of fields $\{\bu^{(k)}\}$ whose defects $\{\bv^{(k)}\}$ converge to zero strongly in $H^1(\Omega; \R^N)$, and not just weakly as in the preceding statement. 

Is there an explicit way to test if a given family of probability measures $\nu=\{\nu_{(t, \bx)}\}_{(t, \bx)\in\Omega}$ complies with Definition \ref{nueva}? This is, we believe, a profound question that can hardly be answered at this stage without further insight. However, in practice a Young measure $\nu=\{\nu_{(t, \bx)}\}_{(t, \bx)\in\Omega}$ is never found neatly as such, but we are to be satisfied with a sequence $\{\bu_j\}$ of nearly solutions, and these can be found through the flow of $E$, or through direct numerical approximations. As a matter of fact, we conjecture that the flow of a variationally hyperbolic system \eqref{ecuaini} will always generate a strong Young measure solution, i.e. we will always have, under the variationally hyperbolic condition, that
$$
E'(\bu^{(j)})\to0\hbox{ implies }E(\bu^{(j)})\to0.
$$
We would like to support the plausability of this conjecture by examining the simple case of a scalar equation in dimension one, where computations can be made much more explicit.

\section{The case of a single equation in dimension one}\label{cinco}
As a way to better understand our proposal, and gain some familiarity with the method, 
it is quite instructive to examine the case of a single conservation law in one space dimension where many facts are better known. 
In particular, it is well-known, under suitable assumptions, that there is a unique entropy solution that can be approximated by the viscosity method. See \cite{bressan}, for instance. 

We hence focus on a non-linear, scalar conservation law
\begin{equation}\label{ecuaini1}
u_t(t, x)+[f(u(t, x))]_x=0\hbox{ in }\Omega\equiv(0, +\infty)\times\R,\quad u(0, x)=u_0(x).
\end{equation}
The unknown is $u(t, x):(0, +\infty)\times\R\to\R$, $f:\R\to\R$ is assumed to be smooth and have polynomial growth of some degree $p\ge1$ at infinity, and $u_0:\R\to\R$ is the function determining the initial value. $u_0$ is 
assumed to belong to $L^2(\R)$. 

Our basic functional space is $\bbX\equiv L^2(\Omega)\cap L^{2p}(\Omega)$. For $u\in\bbX$, define $v(t, x)\in H^1(\Omega)$, which we will identify as its ``defect", as the unique solution of the problem
\begin{equation}\label{primi}
\int_\Omega[ (u+v_t)w_t+(f(u)+v_x)w_x]\,dt\,dx+\int_{\R} u_0(x)w(0, x)\,dx=0
\end{equation}
for all test functions $w\in H^1(\Omega)$. 
As before, we put
\begin{equation}\label{funcer}
E(u)=\frac12\int_\Omega(v_t^2+v_x^2)\,dt\,dx,
\end{equation}
and find that 
$$
E'(u)=-(v_t+f'(u)v_x).
$$

Consider the flow of $E$
\begin{equation}\label{flujo}
\frac {d\bu}{ds}=-E'(\bu)\hbox{ for }t>0,\quad \bu(0)=u^{(0)}(t, x),
\end{equation}
for arbitrary $u^{(0)}\in\bbX$. If we put $u(s, t, x)=\bu(s)$, then 
\begin{equation}\label{flujo}
\frac{du}{ds}=v_t+f'(u)v_x,\quad u(0, t, x)=u^{(0)}(t, x),
\end{equation}
where $v(s, t, x)$ is the corresponding defect for $u(s, t, x)$ so that \eqref{primi} holds for every $s>0$. 

\subsection{The numerical approximation}\label{seis}
The basic property of our error functional allows for a well-founded numerical scheme to approximate weak solutions of the conservation law \eqref{ecuaini}. It is based on a standard steepest-descent strategy just as in \eqref{flujo}. All of the necessary calculations have been performed earlier. It amounts to the following iterative algorithm. It is worthwhile to highlight that the way in which the defect is defined makes a software like FreeFem (\cite{freefem}) an ideal tool to implement this scheme.

\begin{enumerate}
\item Initialization. Take an arbitrary $u^{(0)}\in\bbX$.
\item Main iterative step until convergence. Once $u^{(j)}$ is known,
\begin{enumerate}
\item compute the defect $v^{(j)}$ by solving \eqref{primi} for $u\equiv u^{(j)}$, namely
$$
\int_\Omega[ (u^{(j)}+v^{(j)}_t)w_t+(f(u^{(j)})+v^{(j)}_x)w_x]\,dt\,dx-\int_{\R} u_0w\,dx=0
$$
for all test functions $w\in H^1(\Omega)$;
\item take 
$$
U^{(j)}=v^{(j)}_t+f'(u^{(j)})v^{(j)}_x,
$$
and find an approximation $\epsilon_j$ of $\epsilon>0$ so that $E'(u^{(j)}+\epsilon U^{(j)})=\bcero$; we can take 
$$
\epsilon_j=-\frac{\int_\Omega(v^{(j)}_tV^{(j)}_t+v^{(j)}_xV^{(j)}_x)\,dt\,dx}{\int_\Omega((V^{(j)}_t)^2+(V^{(j)}_x)^2)\,dt\,dx}
$$
where $V^{(j)}$ solves 
$$
\int_\Omega[ (U^{(j)}+V^{(j)}_t)w_t+(f'(u^{(j)})U^{(j)}+V^{(j)}_x)w_x]\,dt\,dx=0,
$$
for all test functions $w\in H^1(\Omega)$;
\item update 
$$
u^{(j+1)}=u^{(j)}+\epsilon_j U^{(j)}.
$$
\end{enumerate}
\end{enumerate}
The somewhat surprising fact is that this numerical scheme does not get stuck in local minima or any other kind of critical point of the error functional $E$ other than a global minimizer: the iterations will steadily lead down the error to zero. The only issue is that, depending on the initial guess, we might be approximating different weak solutions, different strong Young measure solutions, or no solution at all. 

The procedure is so simple to implement that it is worth looking at some numerical experiments. In practice, one can monitor the values $E(u^{(j)})$ to check if they are going down to zero, even if one does not have a proof of this result. The single fact that $E((u^{(j)})\searrow0$, despite further considerations, implies that $u^{(j)}$ is a good approximation of a true weak solution, a strong Young measure solution, or no solution at all in case $\{u^{(j)}\}$ does not remain in a bounded set.

\subsection{Some numerical experiments}
Numerical simulations for simplified situations require a finite domain in the spatial variable $x$, so that boundary condition on the two end-points of the one-dimensional domain are to be enforced appropriately. We focus on typical Riemann problems for Burguer's equation for the quadratic flux function 
$$
f(u)=\frac12u^2.
$$

We will take this time $\Omega=[-1, 1]\times(0, 1)$, taking $T=1$, and $I=(-1, 1)$ as the spatial domain. We would like to find an approximation of the problem
\begin{gather}
u_t+uu_x=0\hbox{ in }\Omega,\nonumber\\
u(0, x)=u_0(x),\quad u(t, -1)=u_L(t),\quad u(t, 1)=u_R(t),\nonumber
\end{gather}
where $u_0$ is the initial value, and $u_{L, R}$ are the boundary data. In the case we take both $u_L$, and $u_R$ as constant values, and
$$
u_0(x)=\begin{cases}u_L,& x<0,\\ u_R,& x>0,\end{cases}
$$
we are solving a typical Riemann problem. 
Instead of taking \eqref{primi} to determine the defect $v$, we take
\begin{align}
\int_\Omega&[ (u+v_t)w_t+(f(u)+v_x)w_x]\,dt\,dx\nonumber\\
&+\int_{t=0} u_0w\,dx-\int_{x=1} f(u_R)w\,dt+\int_{x=-1} f(u_L)w\,dt=0\nonumber
\end{align}
for all test functions $w\in H^1(\Omega)$ with $w(1, x)=0$ for all $x\in I$. Everything else is just like in the situation for all of space. 

We have used values $u_R, u_L$ equal to $\pm 1$. In the case $u_L=-1$, $u_R=1$, the entropy solution is a rarefaction wave. Figure \ref{rar} shows the solution for an initial guess $u^{(0)}\equiv0$. If, on the contrary, we take $u_L=1$, $u_R=-1$, then the entropy solution is, in this case, a stationary shock. For a vanishing initial guess, our above algorithm yields Figure \ref{shock}. In these examples, there is a clear overshooting effect that would have to be corrected for more accurate or more involved simulations. 

\begin{figure}[b]
\includegraphics[scale=0.5]{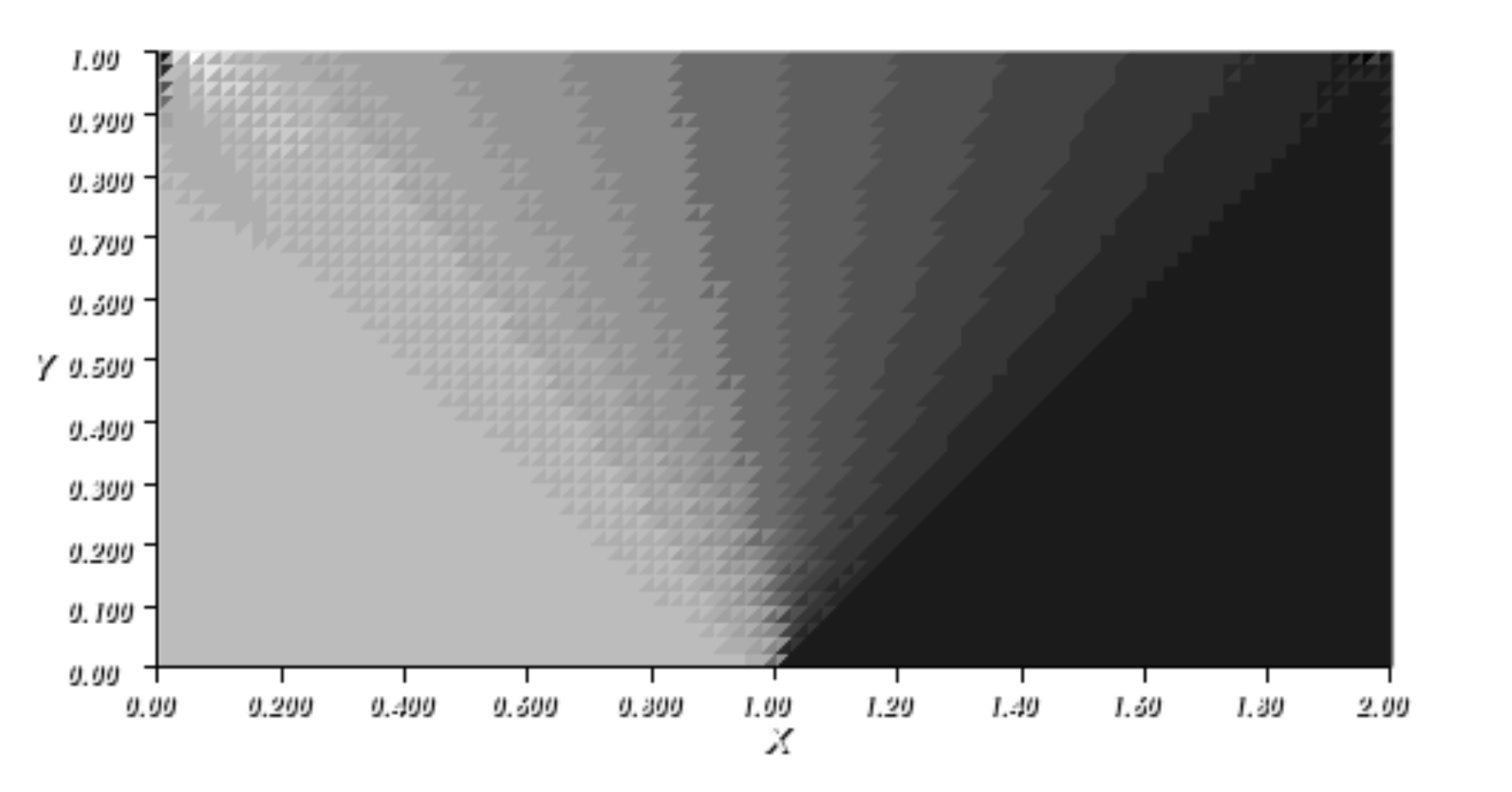}
\caption{A rarefaction wave.}
\label{rar}       
\end{figure}

\begin{figure}[b]
\includegraphics[scale=0.5]{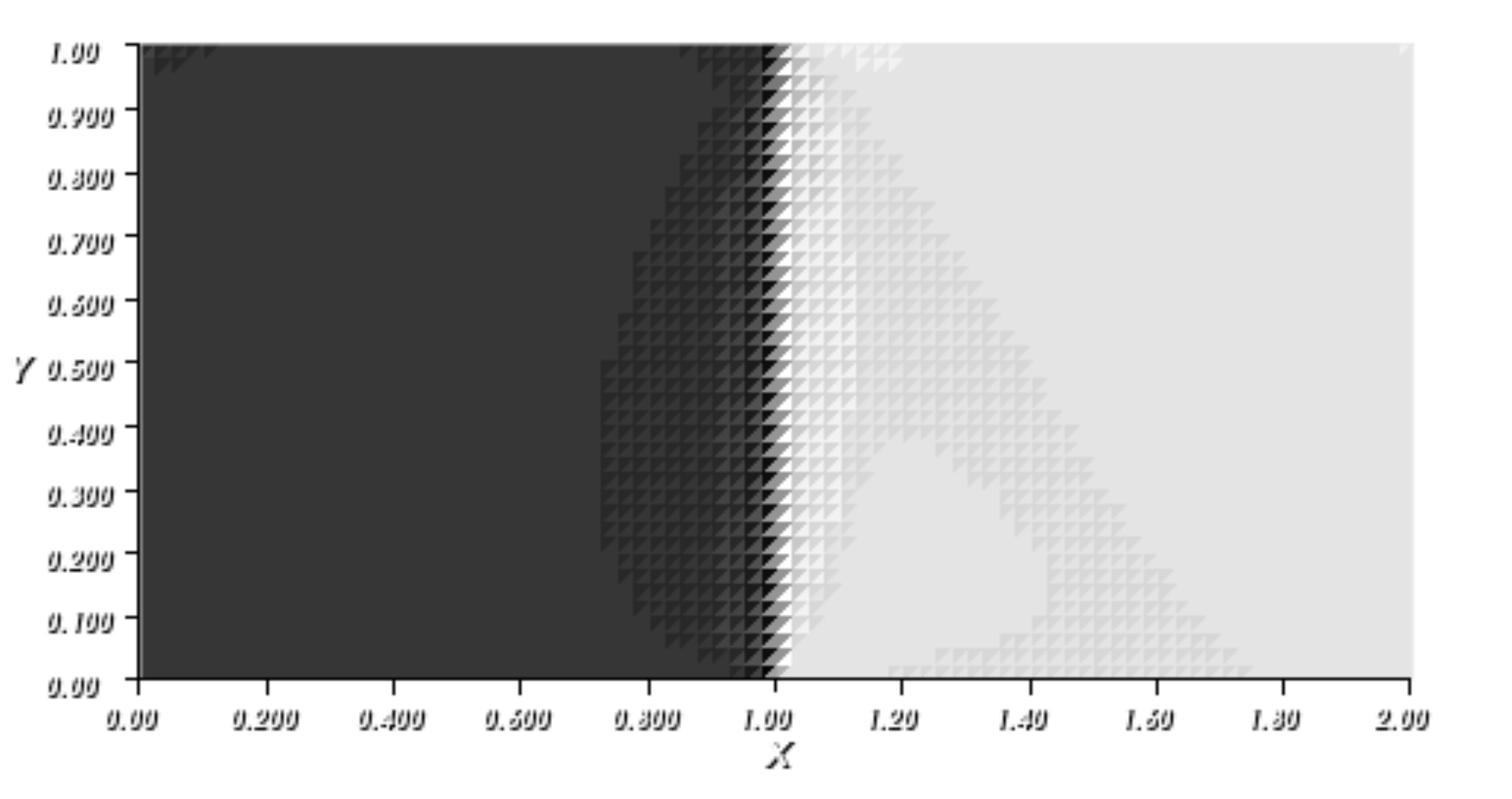}
\caption{A stable shock.}
\label{shock}       
\end{figure}

%
%
%

\subsection{The entropy condition}\label{entropia}
It is well-known that weak solutions of \eqref{ecuaini1}, i.e. satisfying \eqref{primi}, are non-unique. As a matter of fact, there are quite often, infinitely many weak solutions of that kind. A selection criterium needs to be used to choose the physically relevant solution from the full set of solutions. This is well-established through the entropy condition. See \cite{EvansB}, for instance, and references therein. 

Since our approach is variational, we suggest to use a typical variational selection criterium as in analogous situations of non-uniqueness of minimizers. 
Our situation is, at least formally, similar to some of those problems (Cahn-Hilliard models related to phase separation, gradient theory of phase transitions, etc \cite{modica}, \cite{shafrir}). We have a full set of minimizers of a certain error functional $E$, and so we pretend to select the good solutions by proposing a singular perturbation of such functional. Precisely, we will consider the perturbed functional
\begin{equation}\label{perturbacion1}
E_\epsilon(u)=\frac{\epsilon^2}2\int_\Omega [u_t(t, x)^2+u_x(t, x)^2]\,dt\,dx+ E(u)
\end{equation}
for $u\in H^1(\Omega)$ with $u(0, x)=u_0(x)$, so that we assume $u_0\in H^{1/2}(\R)$. Because $E(u)$ is a lower-order functional, and so a compact perturbation of the higher-order term depending on first derivatives, there are always minimizers $u^{(\epsilon)}$ of $E_\epsilon$ for each positive $\epsilon$. 

Because of an innocent mismatch of a minus sign, we will change, in this section, the sign on the definition of the defect $v$ in \eqref{primi}, and set
$$
\int_\Omega[ (u-v_t)w_t+(f(u)-v_x)w_x]\,dt\,dx+\int_{\R} u_0(x)w(0, x)\,dx=0
$$
for all test functions $w\in H^1(\Omega)$. 

\begin{theorem}
Let $u^{(\epsilon)}$ be a branch of minimizers of $E_\epsilon$ over $H^1(\Omega)$ under the initial condition $u(0, x)=u_0(x)$, such that it converges pointwise for  a.e. $(t, x)\in\Omega$ to some $u$. Then $u$ is an entropy solution of our problem \eqref{ecuaini}.
\end{theorem}
\begin{proof}
Let us look at optimality. If $u^{(\epsilon)}$ is a minimizer of $E_\epsilon$ over $H^1(\Omega)$ under the initial condition $u_0(x)$, then for arbitrary $w\in H^1_0(\Omega)$, we will have, according to our previous calculations concerning $E'(u)$ and the extra minus sign, 
\begin{equation}\label{optim}
\int_\Omega [\epsilon^2(u^{(\epsilon)}_tw_t+u^{(\epsilon)}_xw_x)+(v^{(\epsilon)}_t+f'(u^{(\epsilon)})v^{(\epsilon)}_x)w]\,dt\,dx=0,
\end{equation}
where $v^{(\epsilon)}$ is the corresponding defect, and so
$$
\int_\Omega[ (u^{(\epsilon)}-v^{(\epsilon)}_t)w_t+(f(u^{(\epsilon)})-v^{(\epsilon)}_x)w_x]\,dt\,dx+\int_{\R} u_0(x)w(0, x)\,dx=0
$$
for all $w\in H^1(\Omega)$. Since $u^{(\epsilon)}$ belongs to $H^1(\Omega)$, and if we take $w\in H^1_0(\Omega)$ to discard the integral for $t=0$, we can perform an integration by parts in some of the terms, and find that
\begin{equation}\label{compar}
\int_\Omega [u^{(\epsilon)}_tw+v^{(\epsilon)}_tw_t+f'(u^{(\epsilon)})u^{(\epsilon)}_xw+v^{(\epsilon)}_xw_x]\,dt\,dx=0.
\end{equation}
Comparison of \eqref{optim} and \eqref{compar}, due to the arbitrariness of $w\in H^1_0(\Omega)$, leads to $v^{(\epsilon)}=\epsilon u^{(\epsilon)}$, and
$$
\int_\Omega [\epsilon(u^{(\epsilon)}_tw_t+u^{(\epsilon)}_xw_x)+(u^{(\epsilon)}_t+f'(u^{(\epsilon)})u^{(\epsilon)}_x)w]\,dt\,dx=0.
$$
The minimizer $u^{(\epsilon)}$ is then a solution of
\begin{equation}\label{visc}
-\epsilon(u^{(\epsilon)}_{tt}+u^{(\epsilon)}_{xx})+u^{(\epsilon)}_t+f(u^{(\epsilon)})_x=0\hbox{ in }\Omega,
\end{equation}
with $u^{(\epsilon)}(0, x)=u_0(x)$. It is then standard, despite having a second time derivative in the approximating equation (see \cite{EvansB} for instance), to check that the limit $u$ is an entropy solution of the problem. Indeed, take an entropy pair $(\phi, \psi)$
$$
\phi'(z)f'(z)=\psi'(z),\quad \phi(z),\hbox{ convex}.
$$
Then standard calculations, taking advantage of \eqref{visc}, yield
\begin{align}
\phi(u^{(\epsilon)})_t+\psi(u^{(\epsilon)})_x=&\epsilon\phi'(u^{(\epsilon)})(u_{tt}^{(\epsilon)}+u_{xx}^{(\epsilon)})\nonumber\\
=&\epsilon[\phi(u^{(\epsilon)})_{tt}+\phi(u^{(\epsilon)})_{xx}]\nonumber\\
&-\epsilon[\phi''(u^{(\epsilon)})(u_t^{(\epsilon)})^2+\phi''(u^{(\epsilon)})(u_x^{(\epsilon)})^2].\nonumber
\end{align}
The last term is non-positive because $\phi$ is convex, and hence if $w$ is a smooth, compactly-supported, non-negative test function, we find
$$
\int_\Omega w[\phi(u^{(\epsilon)})_t+\psi(u^{(\epsilon)})_x]\,dt\,dx\ge -\epsilon\int_\Omega \phi(u^{(\epsilon)})(v_{tt}+v_{xx})\,dt\,dx.
$$
As $\epsilon\searrow0$, we obtain the entropy inequality. 
This is standard.
\end{proof}

Our discussion here about the entropy condition supports the fact that the approximated solution through the scheme of the preceding subsection will always, regardless of the initial guess $u^{(0)}(t, x)$, converge to the true entropy condition. This is so simply because the discretization of the problem always introduces a small numerical viscosity which mimics the singularly perturbed functional $E_\epsilon$ in \eqref{perturbacion1}.



\begin{thebibliography}{99}
\bibitem{ball}  Ball, J. M., A version of the fundamental theorem for Young measures. PDEs and continuum models of phase transitions (Nice, 1988), 207?215, Lecture Notes in Phys., 344, Springer, Berlin, 1989. 
\bibitem{benzoni-serre} Benzoni-Gavage, S.; Serre, D., Multidimensional hyperbolic partial differential equations. First-order systems and applications. Oxford Mathematical Monographs. The Clarendon Press, Oxford University Press, Oxford, 2007.
\bibitem{bogunz} Bochev, Pavel B.; Gunzburger, Max D. Least-squares finite element methods. Applied Mathematical Sciences, 166. Springer, New York, 2009.
\bibitem{brenner1} Brenner, P., The Cauchy problem for symmetric hyperbolic systems in $L_p$, Math. Scand., 19 (1966), 27-37.
\bibitem{brenner2} Brenner, P., The Cauchy problem for systems in $L_p$ and $L_{p, \alpha}$, Ark. Mat., 11 (1973), 75-101.
\bibitem{bressan} Bressan, A., {\it Hyperbolic Systems of Conservation Laws}, Oxford Lectures Series Math. Appl., 20, New York, 2000.
\bibitem{dst}   Demoulini, S., Stuart, D. M. A., Tzavaras, A. E. Weak-strong uniqueness of dissipative measure-valued solutions for polyconvex elastodynamics. Arch. Ration. Mech. Anal. 205 (2012), no. 3, 927?961.
\bibitem{diperna}  DiPerna, R. J. Measure-valued solutions to conservation laws, Arch. Rational Mech. Anal. 88 (1985), no. 3, 223?270.
\bibitem{EvansB} Evans, L. C., {\it Partial Differential Equation}, Grad. Studies Math., Volume 19, AMS, 2010 (second edition), Providence. 
\bibitem{feireisl} Chiodaroli, E., Feireisl, E., Kreml, O., Wiedemann, E., $\mathcal A$-free rigidity and applications to the compressible Euler system, Annali di Matematica Pura ed Applicata, (to appear).
\bibitem{tadmor2}  Fjordholm, U. S., Mishra, S., Tadmor, E., On the computation of measure-valued solutions. Acta Numer. 25 (2016), 567?679.
\bibitem{tadmor} Fjordholm, U. S., K\"appeli, R., Mishra, S., Tadmor, E., Construction of approximate entropy measure valued solutions for hyperbolic systems of conservation laws, arXiv:1402.0909v2.
\bibitem{freefem} Hecht, F., New development in freefem++, J. Numer. Math. 20 (2012), no. 3-4, 251-265. 
\bibitem{glowinski} R. Glowinski, Numerical methods for Nonlinear Variational Problems, Springer Series in Computational Physics, 1983. 
\bibitem{lax}  Lax, P. D., Hyperbolic systems of conservation laws. II. Comm. Pure Appl. Math. 10 1957 537?566. 
\bibitem{LionsC} Lions, J. L. 1971 \textit{Optimal Control of Systems governed by Partial
Differential Equations}, Springer.
 \bibitem{mnrr} Málek, J.; Ne?as, J.; Rokyta, M.; R?¸i?ka, M., Weak and measure-valued solutions to evolutionary PDEs. Applied Mathematics and Mathematical Computation, 13. Chapman \& Hall, London, 1996.
\bibitem{modica} Modica, L., Gradient theory of phase transitions with boundary contact energy. Ann. Inst. H. Poincaré Anal. Non Linéaire 4 (1987), no. 5, 487-512.
\bibitem{neustupa} Neustupa, J., Measure-valued solutions of the Euler and Navier-Stokes equations
for compressible barotropic fluids, Math. Nachr. 163, (1993), 217-227.
\bibitem{ped4} Pedregal, P., Parametrized measures and variational principles. Progress in Nonlinear Differential Equations and their Applications, 30. Birkhäuser Verlag, Basel, 1997.
\bibitem{ped1} Pedregal, Pablo On error functionals. Se? MA J. 65 (2014), 13?22.
\bibitem{ped2} Pedregal, P., Some remarks on non-linear elliptic PDEs, (submitted).
\bibitem{shafrir} Shafrir, I., On a class of singular perturbation problems, Handbook of Differential Equations, Stationary Partial Differential Equations, Vol. 1, ed. by M. Chipot and P. Quittner, 297-384, Elsevier. 
\bibitem{szekelyhidi} Székelyhidi, L. Wiedemann, E. Young measures generated by ideal incompressible fluid flows. Arch. Ration. Mech. Anal. 206 (2012), no. 1, 333?366.
\bibitem{tartar0} Tartar, L., Compensated compactness and applications to partial differential equations. Nonlinear analysis and mechanics: Heriot-Watt Symposium, Vol. IV, pp. 136?212, Res. Notes in Math., 39, Pitman, Boston, Mass.-London, 1979. 
\bibitem{tartar}  Tartar, L., From hyperbolic systems to kinetic theory. A personalized quest. Lecture Notes of the Unione Matematica Italiana, 6. Springer-Verlag, Berlin; UMI, Bologna, 2008.
\end{thebibliography}
\end{document}